\documentclass[12pt,oneside]{amsart}

      \usepackage{amssymb}
      \usepackage[english]{babel}
      \usepackage{graphicx}
      \usepackage[applemac]{inputenc}
      \usepackage{float}
      \usepackage{wrapfig}
      \usepackage[a4paper=true,%
      breaklinks=true,%
      colorlinks=true,%
      linkcolor={blue},%
      citecolor={red},%
      pdftitle={A note on the real part of Riemann's zeta function}%
      pdfauthor={Juan Arias de Reyna, Richard P. Brent, Jan van de Lune},%
      pdfkeywords={snapshot},%
      pdfstartview=FitH,
       ]{hyperref}

      \theoremstyle{plain}
      \newtheorem{theorem}{Theorem}[section]
      \newtheorem{lemma}[theorem]{Lemma}

      \theoremstyle{definition}

      \theoremstyle{remark}

      \newcommand{\C}{{\mathbb C}}
      \newcommand{\R}{{\mathbb R}}

      \newfont{\cmbsy}{cmbsy10}
      \newfont{\cmmib}{cmmib10}

      \def\Re{\mathrm{Re\,}}

      \makeatletter
      \def\@setcopyright{}
      \def\serieslogo@{}
      \makeatother

\begin{document}

   \author{Juan Arias de Reyna}
   \address{Universidad de Sevilla, Facultad de Matem\'aticas,
	    Apdo.\ 41080-Sevilla, Spain}
   \email{arias@us.es}
   
   \author{Richard P. Brent}

   \address{Mathematical Sciences Institute, 
   Australian National University, Canberra, ACT 0200, Australia}
   \email{abl@rpbrent.com}

   \author{Jan van de Lune }
   \address{\noindent Langebuorren 49, 9074 CH Hallum, The Netherlands 
   \newline(Formerly at CWI, Amsterdam ) }
   \email{j.vandelune@hccnet.nl}


   \title[On the real part of the Riemann zeta-function]
   {A note on the real part of the Riemann zeta-function}


   \begin{abstract}
   We consider the real part $\Re\zeta(s)$ of the Riemann zeta-function
   $\zeta(s)$ in the half-plane $\Re(s) \ge 1$. We show how to compute
   accurately the constant $\sigma_0 \approx 1.19$ 
   which is defined to be the supremum of $\sigma$ such
   that $\Re\zeta(\sigma+it)$ can be negative (or zero)
   for some real $t$.
   We also consider intervals where $\Re\zeta(1+it) \le 0$ 
   and show that they are rare. The first occurs for $t \approx 682112.9$,
   and has length $\approx 0.05$.  
   We list the first $50$ such intervals.
   \end{abstract}

\vspace*{-10pt}




   \dedicatory{Dedicated to Herman J.~J.~te Riele on the occasion
     of his retirement\\ from the CWI in January 2012}


   \maketitle

\vspace*{-10pt}

\section{Introduction}

\thispagestyle{empty}

   In this note we consider the real part of the Riemann zeta-function
   $\zeta(s)$ in the half-plane $H = \{s \in \C\,|\,\Re(s) \ge 1\}$. 
   As usual, we write
   $s = \sigma + it$, so $\Re(s) = \sigma \ge 1$. We are mainly interested
   in the regions where $\Re\zeta(s) \le 0$.
   Since $\lim_{\sigma\uparrow\infty}\zeta(\sigma+it)=1$ 
   (uniformly in $t$), $\Re\zeta(\sigma+it)$
   cannot be zero for arbitrarily large $\sigma>1$.
   We define \[\sigma_0 := \sup \{\sigma \in \R\,|\, (\exists t \in \R)\,
		\Re\zeta(\sigma+it) = 0\}.\]

   Thus, $\Re\zeta(s) > 0$ if $\sigma > \sigma_0$.
   In van de Lune \cite{L} it was shown that $\sigma_0$
   is the (unique) positive real root of the equation
   \begin{displaymath}
   \sum_p\arcsin\Bigl(\frac{1}{p^\sigma}\Bigr) = \frac{\pi}{2}\,,
   \end{displaymath}
   where $p$ runs through the primes (we adopt this convention throughout).
   In \cite{L} it was also shown that $\sigma_0>1.192$ and that
   $\Re\zeta(\sigma_0 + it)$ never vanishes. 

   The main aim of this note is to show how $\sigma_0$ can be computed to 
   arbitrarily high precision by an efficient algorithm. 
   We also mention some results on the behaviour of
   $\Re\zeta(\sigma+it)$ for $1 \le \sigma < \sigma_0$, and in particular
   on the line $\sigma = 1$. 

\section{Accurate computation of the constant $\sigma_0$} \label{sec:theory}

   In this section we assume that $\sigma \ge \sigma_1 > 1$, where 
   $\sigma_1$ is a suitable constant (e.g. $1.1$).
   We show how the constant $\sigma_0$ can be computed within a given
   error bound. There are three main steps.

   \begin{enumerate}
   \item Give an algorithm to evaluate the 
   {\em prime zeta-function} \cite{Froberg}
	\[P(\sigma) = \sum_p p^{-\sigma},\] for real $\sigma > 1$.
   \item Using step 1, 
   give an algorithm to evaluate the function $f(\sigma)$ defined by
   \begin{displaymath}
   f(\sigma) = \sum_p\arcsin\Bigl(\frac{1}{p^\sigma}\Bigr) - \frac{\pi}{2}.
   \end{displaymath}
   \item Use a suitable zero-finding algorithm to locate a zero of $f(\sigma)$
   in a (sufficiently small) interval where $f(\sigma)$ changes sign,
   for example $[1.1, 1.2]$.
   \end{enumerate}

   Step 1 is easy.  From the Euler product for $\zeta(\sigma)$ and M\"obius 
   inversion, we have a formula essentially known to Euler \cite[1748]{Euler}:
   \begin{equation}
   \label{eq:P}
   P(\sigma) = \sum_{r=1}^\infty\frac{\mu(r)}{r}\log\zeta(r\sigma),
   \end{equation}
   which is valid for $\sigma > 1$
   (see Titchmarsh \cite[eqn.\ (1.6.1)]{T}).
   The series converges rapidly in view of the following Lemma.
   \begin{lemma}	\label{lemma:P}
   For $\sigma \ge 2$, $0 < \log \zeta(\sigma) < 3/2^{\sigma}$
   and $0 < P(\sigma) < 3/2^{\sigma}$.
   \end{lemma}
   \begin{proof}
   For $\sigma \ge 2$, we have
   \[ 0 < \zeta(\sigma) - 1 < 2^{-\sigma} + 3^{-\sigma}
	+ \int_3^\infty x^{-\sigma}{\rm d}x
	= 2^{-\sigma} + 3^{-\sigma} + \frac{3^{1-\sigma}}{\sigma-1}
	< 3/2^{\sigma}, \]
   so
   \[ 0 < \log\zeta(\sigma) < \zeta(\sigma) - 1 < 3/2^{\sigma}.\]
   The upper bound on $P(\sigma)$ follows similarly, using
   $P(\sigma) < \zeta(\sigma) - 1$.
   \end{proof}
   Using (\ref{eq:P}) and Lemma~\ref{lemma:P}, we have
   \[P(\sigma) = \log\zeta(\sigma) +
	\sum_{r=2}^\infty\frac{\mu(r)}{r}\log\zeta(r\sigma),\]
   where the $r$-th term in the sum is bounded in absolute value
   by $3/2^{r\sigma+1}$. Thus, we can evaluate $P(\sigma)$ accurately,
   for given $\sigma > 1$, using any good algorithm for the evaluation of
   $\zeta(\sigma)$, for example Euler-Maclaurin summation.
   If (\ref{eq:P}) is used to compute $P(\sigma)$, $P(3\sigma)$,
   $P(5\sigma), \ldots,$ then we should take care to compute the relevant
   terms $\log\zeta(r\sigma)$ only once.

   For step 2, we observe that the $\arcsin$ series defining $f(\sigma)$
   converges slowly and irregularly, since it is a sum over primes which to
   first order behaves like $\sum_p p^{-\sigma}$.  The well-known ``trick''
   is to express $f(\sigma)$ as a double series and reverse the order of
   summation, obtaining an expression which is mathematically equivalent but
   computationally far superior.  For some similar examples, see 
   Wrench \cite[1961]{Wrench}.
 
   For $|x|<1$ we have 
   \begin{displaymath}
   \arcsin(x)=\sum_{k=0}^\infty c_k x^{2k+1},
   \end{displaymath}
   where 
   \begin{displaymath}
   c_k=\frac{1\cdot3\cdot5\cdots(2k-1)}{2\cdot 4\cdot6\cdots(2k)}
   \frac{1}{2k+1}=\frac{(2k)!}{(2^kk!)^2(2k+1)}\qquad\text{for $k\ge0$}.
   \end{displaymath}
   Note that all $c_k$ are positive so that $f(\sigma)$ is strictly convex.
   It is also clear that $f(\sigma)$ is strictly
   decreasing for $\sigma>1$. 
   From the expression for $c_k$, we see that, for $k \ge 1$,
   \begin{equation}
   c_k \le \frac{1}{2(2k+1)}\,.			\label{eq:cineq}
   \end{equation} 

   For $\sigma > 1$ it is easy to justify interchanging the
   order of summation in
   \begin{equation*}
   f(\sigma)=\sum_p\sum_{k=0}^\infty c_k\Bigl(\frac{1}{p^\sigma}\Bigr)^{2k+1}
	- \; \frac{\pi}{2}\,,
   \end{equation*}
obtaining
   \begin{equation}		\label{eq:fseries}
   f(\sigma) =\sum_{k=0}^\infty c_k\sum_p\frac{1}{p^{(2k+1)\sigma}}
	- \frac{\pi}{2}
   =\sum_{k=0}^\infty c_k P\bigl((2k+1)\sigma\bigr)
	- \frac{\pi}{2}\,.
   \end{equation}

   From Lemma~\ref{lemma:P} and the inequality (\ref{eq:cineq}),
   we see that
   \[
   0 < 
   \sum_{k=K+1}^\infty c_k P\bigl((2k+1)\sigma\bigr) < 2^{-(2K+3)\sigma},
   \]
   so it is easy to determine $K$ such that we can truncate the series
   in (\ref{eq:fseries}) to a finite sum over $k \le K$ with a rigorous
   error bound.

   If desired, we can substitute (\ref{eq:P}) into (\ref{eq:fseries})
   and interchange the order of summation, obtaining\footnote{We thank
Charles Voas for pointing out an error in equation~\eqref{eq:fseries2}
as stated in earlier versions of this paper. Fortunately, this error 
did not affect the computational results, 
which were obtained using~\eqref{eq:fseries}.}
   \begin{equation}	\label{eq:fseries2}
   f(\sigma) = \sum_{j=1}^\infty d_j \log \zeta(j\sigma)
     - \frac{\pi}{2},
   \end{equation}
   where
   \[
   d_j = \sum_{k\ge 0,\, r > 0,\, (2k+1)r = j}
		\frac{c_k\mu(r)}{r}\,.
   \]
   From the inequality $c_k \le 1/(2k+1)$ (valid for $k \ge 0$),
   it follows that $|d_j| \le 1$.
   Using Lemma \ref{lemma:P}, we can determine where to safely
   truncate the series (\ref{eq:fseries2}).

   For step 3, we can use a zero-finding algorithm which needs only
   function (not derivative) evaluations, and gives a guaranteed bound
   on the final result.  For example, the method of bisection could be
   used, but would be slow, taking about $\log_2(1/\varepsilon)$
   function evaluations to obtain a solution with error bounded
   by $\varepsilon$.  In the secant method, a sequence $(x_n)$, converging
   to a zero of $f$ under suitable conditions, is obtained by
   computing the approximation
   $x_{n}$ by linear interpolation using the two
   points
   $(x_{n-1}, f(x_{n-1}))$ and $(x_{n-2}, f(x_{n-2}))$. 
   It converges with order
   $(1+\sqrt{5})/2 \approx 1.618$, but does not always give a
   guaranteed bound on the error.  A combination of bisection
   and linear interpolation, as in the algorithms
   of Dekker~\cite{Dekker} or Brent~\cite{Br},
   can give convergence about as fast as the secant method, but 
   with the final result bracketed in a short interval where the function $f$
   changes sign.

   \section{Computational results}

   The second and third authors independently wrote programs implementing
   the ideas of \S\ref{sec:theory}, using Magma in one case and
   Mathematica 4 and 8 in the other case. The programs used different
   strategies to obtain a final interval where $f$ changes sign (in one
   case taking advantage of the strict convexity of $f$). 
   The output of the programs agrees to at least 500D.
   We give here the correctly rounded result to 100D:

   \begin{displaymath}\begin{split}
   \sigma_0\approx
   1.&19234\,73371\,86193\,20289\,75044\,27425\,59788\,34011\,19230\,
   83799\\[-4pt]
   &94301\,37194\,92990\,52458\,64848\,30139\,24084\,99863\,83788\,36244\,.\\
   \end{split}\end{displaymath}
   Programs and higher precision values are available from the authors.

   \section{The distribution of $\Re\zeta(\sigma+it)$ for $\sigma \ge 1$}

   Assuming that the limit exists, we define
   \[
   d(\sigma) = \lim_{T \to +\infty}
		\frac{1}{T}\,m\{t\in [0,T]\,|\, \Re\zeta(\sigma + it) < 0\},
   \]
   where $m$ denotes Lebesgue measure.
   Informally, $d(\sigma)$ is the probability that $\zeta(s)$ has
   negative real part on a given vertical line $\Re(s) = \sigma$.

   The results of Section \ref{sec:theory} show that
   $d(\sigma) = 0$ for $\sigma \ge \sigma_0 \approx 1.19$.
   Here we briefly discuss the region $1 \le \sigma < \sigma_0$.

   At least for those values of $t$ that are accessible to computation, 
   $\Re\zeta(\sigma+it)$ is ``usually'' positive for $\sigma \ge 1$.
   The function
   $d(\sigma)$ is conjectured to be continuous and monotonic decreasing
   from a positive value at $\sigma = 1$ to zero at $\sigma = \sigma_0$.
   Even on the line $\sigma = 1$, $\Re\zeta(\sigma+it)$ is usually positive
   \cite{Milioto}.
   We can prove that $d(1) < 1/4$, but a Monte Carlo computation
   suggests that the true value is much smaller. Based on $5 \times 10^{11}$
   pseudo-random trials, we estimate $d(1) = (3.80 \pm 0.01) \times 10^{-7}.$
   Similarly, we estimate
   $d(1.01) = (1.10\pm 0.01) \times 10^{-7}$ and
   $d(1.02) \approx (2.66\pm 0.04) \times 10^{-8}$, so it can be seen that 
   $d(\sigma)$ decreases rapidly as we move to the right of $\sigma = 1$.

   Although $\zeta(s)$ has a simple pole at $s = 1$, the Laurent series
  \[\zeta(s) = \frac{1}{s-1} + \gamma + O(|s-1|)\]
   shows that $\Re\zeta(1+it)$ 
   has a positive limit $\gamma = 0.577\cdots$ (Euler's constant) 
   as $t \to 0$.

   On any fixed vertical line $\sigma > 1$, both
   $\zeta(\sigma+it)$ and $1/\zeta(\sigma+it)$ are bounded,
   in fact
   ${\zeta(2\sigma)}/{\zeta(\sigma)} < |\zeta(\sigma + it)|
	\le \zeta(\sigma)$.
   However, the situation is different on the line $\sigma = 1$,
   as both $\zeta(1+it)$ and $1/\zeta(1+it)$ are unbounded. Their 
   true order of growth is unknown.
   It follows from Titchmarsh~\cite[Theorem 11.9]{T} and the continuity
   of $\Re\zeta(1+it)$ that
   $\Re\zeta(1+it)$ attains all real values. 
   Nevertheless, the ``usual''
   values are quite small. As a special case of
   \cite[Theorem 7.2]{T} we have the mean value theorem
   \[\lim_{T \to \infty} \frac{1}{T}\,\int_1^T |\zeta(1+it)|^2\, {\rm d}t
     = \zeta(2) = \frac{\pi^2}{6}\,.\] 
   Using ideas as in the proof of \cite[Theorem 7.2]{T},
   we can prove that
   \[\lim_{T \to \infty} \frac{1}{T}\,\int_0^T \Re\zeta(1+it)\, {\rm d}t
     = 1\,.\] 
   Thus, informally, we can say that the typical value of $\Re\zeta(1+it)$
   is close to~$1$. The values have a distribution with mean $1$
   and variance $\pi^2/6 - 1 \approx 0.645$.

   In~\cite[Table 1]{L}, van de Lune gave a list of values 
   of $t > 0$ such that
   $\Re\zeta(1+it) < 0$ and is (approximately) a local minimum. 
   The list was not claimed to be exhaustive.  The smallest $t$ listed
   was $t = 682112.92$ with $\Re\zeta(1+it) \approx -0.003$. We have shown,
   using the ``maximum slope principle'' \cite{LR}, 
   that this is very close to the smallest $t$ for which 
   $\Re\zeta(1+it) \le 0$.  More precisely,
   $\Re\zeta(1+it) > 0$ for $0 < t < 682112.8913$,
   and there is a local minimum of $-.0027652$ at $t \approx 682112.9169$.
   In applying the maximum slope principle we used the bound
   \[
    \left|\frac{\rm d}{{\rm d}t}\arg\zeta(1+it)\right| =
    \left|\Re\frac{\zeta'(1+it)}{\zeta(1+it)}\right| \le
    \frac{3}{4}\,\log(t^2+4) + 7 \;\; {\rm for} \;\; t \ge 10. \]

   \begin{table}[tp] 
   \centering
   \caption{First $50$ negative local minima of $\Re\zeta(1+it)$}
   \label{table:one}
   \begin{tabular}{|c|c|c||c|c|c|}
   \hline
   $t$ & $\Re\zeta$ & length &
   $t$ & $\Re\zeta$ & length \\
   \hline
    $\;\;$
   $  682112.9169$ & $-0.0028$ & $0.0529$&
   $ 8350473.4853$ & $-0.0019$ & $0.0451$\\
   $ 1267065.1710$ & $-0.0040$ & $0.0655$&
   $ 8366684.0439$ & $-0.0197$ & $0.1322$\\
   $ 1466782.0667$ & $-0.0013$ & $0.0391$&
   $ 8452317.9526$ & $-0.0090$ & $0.0900$\\
   $ 1858650.0915$ & $-0.0282$ & $0.1686$&
   $ 8967566.5926$ & $-0.0148$ & $0.1336$\\
   $ 2023654.7671$ & $-0.0221$ & $0.1389$&
   $ 9960968.8748$ & $-0.0184$ & $0.1373$\\
   $ 2064996.2141$ & $-0.0117$ & $0.1076$&
   $11231380.7309$ & $-0.0099$ & $0.1042$\\
   $ 2195056.7909$ & $-0.0755$ & $0.2718$&
   $11236680.3350$ & $-0.0262$ & $0.1595$\\
   $ 2202620.3296$ & $-0.0111$ & $0.1159$&
   $11781932.0257$ & $-0.0170$ & $0.1288$\\
   $ 2530662.6360$ & $-0.0072$ & $0.0865$&
   $11884021.9776$ & $-0.0035$ & $0.0564$\\
   $ 3259774.5293$ & $-0.0471$ & $0.2098$&
   $12045289.3337$ & $-0.0644$ & $0.2498$\\
   $ 3548283.4160$ & $-0.0189$ & $0.1459$&
   $12276788.1573$ & $-0.0182$ & $0.1476$\\
   $ 4052438.9330$ & $-0.0023$ & $0.0474$&
   $12546625.7916$ & $-0.0455$ & $0.2031$\\
   $ 4197235.0783$ & $-0.0331$ & $0.1977$&
   $12781127.5748$ & $-0.0102$ & $0.0964$\\
   $ 5410820.7150$ & $-0.0008$ & $0.0307$&
   $13598773.5889$ & $-0.0543$ & $0.2317$\\
   $ 6027913.8513$ & $-0.0181$ & $0.1325$&
   $13786262.5457$ & $-0.0826$ & $0.2635$\\
   $ 6164063.0008$ & $-0.0263$ & $0.1603$&
   $13922411.7750$ & $-0.0222$ & $0.1418$\\
   $ 6238849.4877$ & $-0.0071$ & $0.0827$&
   $14190358.4974$ & $-0.0632$ & $0.2214$\\
   $ 6265907.4688$ & $-0.0030$ & $0.0522$&
   $14391623.0217$ & $-0.0016$ & $0.0437$\\
   $ 6421627.2235$ & $-0.0241$ & $0.1651$&
   $14788310.5330$ & $-0.0149$ & $0.1132$\\
   $ 7338152.4379$ & $-0.0043$ & $0.0656$&
   $14856540.3430$ & $-0.0220$ & $0.1442$\\
   $ 7469838.9709$ & $-0.0009$ & $0.0305$&
   $15173904.7533$ & $-0.0041$ & $0.0800$\\
   $ 7766995.0303$ & $-0.0742$ & $0.2840$&
   $15321273.7219$ & $-0.0131$ & $0.1181$\\
   $ 7774558.3985$ & $-0.0672$ & $0.2705$&
   $16083163.0244$ & $-0.0098$ & $0.1038$\\
   $ 7985493.9836$ & $-0.0324$ & $0.1728$&
   $16503899.3235$ & $-0.0060$ & $0.0680$\\
   $ 8299958.2327$ & $-0.0022$ & $0.0432$&
   $16656258.8346$ & $-0.0155$ & $0.1329$\\
   \hline
   \end{tabular}
   \end{table}

   Table~\ref{table:one} lists the first $50$ local
   minima of $\Re\zeta(1+it)$ for which\linebreak 
   $t > 0$ and $\Re\zeta(1+it) \le 0$ (no minima are exactly zero).
   The values  
   in the table are rounded to $4$ decimal places.
   The columns headed ``length'' give the lengths of the intervals 
   containing $t$
   in which $\Re\zeta$ is negative. To $8$ decimal places, the first interval,
   of length $0.05291225$, is $(682112.89133824, 682112.94425049)$. 
   The sum of the lengths of the first $50$ intervals is 
   $6.48390168$, 
   giving an estimate $d(1) \approx 3.85\times 10^{-7}$.
   This is close to our Monte Carlo 
   estimate $d(1) \approx 3.80 \times 10^{-7}$.

   In this brief note we refrain from commenting on the region\linebreak
   $\sigma \in [1/2, 1)$, but refer the interested reader to the literature, 
   such as Bohr and Jessen \cite{BJ}, Titchmarsh \cite[\S11.13]{T},
   Tsang \cite{Tsang}, Joyner \cite{Joyner}, 
   Laurin\v{c}ikas \cite{Laurincikas}, Steuding \cite{Steuding}
   and K\"uhn \cite{Kuhn}.

\end{document}